\documentclass[letterpaper, reqno]{amsart}
\usepackage{amsmath,amsthm,amssymb,mathtools}
\usepackage{tabularx,epsfig}
\begin{document}

\hfuzz4pt 

\setlength{\floatsep}{15pt plus 12pt}
\setlength{\textfloatsep}{\floatsep}
\setlength{\intextsep}{\floatsep}
\setlength{\intextsep}{\floatsep}

\newcommand{\mypdf}[3]{\begin{figure}[htbp]\begin{center}
      {\scalebox{#1}{\includegraphics{#2.pdf}}}
      \caption{#3}\label{fig:#2}
    \end{center}\end{figure}}
    
\makeatletter
\@namedef{subjclassname@2020}{1980 Mathematics Classification}
\makeatother

\theoremstyle{plain}
\newtheorem{prop}{Proposition}

\title[R-S Algorithms Obtained from Tableau Recursions]{ROBINSON-SCHENSTED
ALGORITHMS\\
obtained from\\TABLEAU RECURSIONS}

\date{June 2, 1987}

\begin{abstract}
The numbers $f_\lambda$ of standard tableaux of shape $\lambda\vdash n$
satisfy 2 fundamental recursions: $f_\lambda = \sum f_{\lambda^-}$ and
$(n + 1)f_\lambda=\sum f_{\lambda^+}$,
where $\lambda^-$ and $\lambda^+$ run over all shapes
obtained from $\lambda$ by adding or removing a square respectively. The first
of these recursions is trivial; the second can be proven algebraically from the
first. These recursions together imply algebraically the dimension formula
$n! =\sum f_\lambda^2$ for the irreducible representations of $S_n$. We show
that a combinatorial analysis of this classical algebraic argument produces an infinite
family of algorithms, among which are the classical Robinson-Schensted
{\it row\/} and {\it column\/} insertion algorithms. Each of our algorithms yields a
bijective proof of the dimension formula. 
\end{abstract}

\author{Adriano~M.~Garsia}

\address{Author's address:
  Department of Mathematics, University of California, San Diego, La Jolla, CA
92093, USA}
\thanks{Work supported by NSF grant at the Univ.~of Cal.~San Diego and by ONR
grant at the Mass.~Inst.~of Tech.}

\author{Timothy~J.~McLarnan}

\address{Author's current address:
  Department of Mathematics, University of Michigan, Ann Arbor, MI 48105-1003}

\maketitle
 
\section* {Introduction}
 
The Robinson-Schensted row insertion algorithm (briefly RSA) yields a
correspondence (briefly RSC) between permutations and pairs of tableaux of the
same shape. It was discovered independently by Robinson [4] in the context of
group representation theory and over twenty years later by Schensted [6] in the
context of {\it sorting.} It was further studied and extended by Knuth [1,2] and
by Schützenberger~[7,8], among others. It has become a standard tool in the
combinatorial study of representations of the symmetric group $S_n$ and of
symmetric polynomials.

One way to see the representation theoretic significance of the RSA is to think
of it as a bijection proving the identity 
\begin{equation}
n!=\sum_{\lambda\vdash n}f_\lambda^2\;.\tag{I.1}
\end{equation}
The sum here is over all partitions $\lambda$ of $n$, and $f_\lambda$ denotes
the number of standard tableaux of shape $\lambda$. That is, the RSA takes any
permutation and produces from it a pair of standard tableaux of the same shape.
Since $f_\lambda$ is the dimension of the irreducible representation of $S_n$
corresponding to the partition $\lambda$, (I.1) is the specialization to $S_n$
of the general fact that if $\{A^\lambda\}$ are the inequivalent irreducible
representations of $G$, then
\begin{equation}
|G|=\sum\text{dim}^2(A^\lambda)\;.\tag{I.2}
\end{equation}

Formula (I.1) was known well before the development of the RSA. In particular,
Young [9] gave an algebraic proof which is repeated somewhat more readably by
Rutherford [5]. This proof begins with the fundamental recursion 
\begin{equation}
f_\lambda=\sum_{\lambda^-}f_{\lambda^-}\;.\tag{I.3}
\end{equation}
Here $\lambda^-$ varies over all partitions of $n - 1$ obtained by decreasing
one of the parts of 
$\lambda\vdash n$ by 1. From this downward recursion, Young proves the upward
recursion 
\begin{equation}
(n+1)f_\lambda=\sum_{\lambda^+}f_{\lambda^+}\;,\tag{I.4}
\end{equation}
where $\lambda^+$ ranges over partitions obtained by incrementing some part of
$\lambda$ by~1. These two recursions together imply (I.1). 

We shall show that a straightforward combinatorial analysis of this algebraic
proof produces the Robinson-Schensted algorithm. In fact, the argument outlined
above is exactly an algebraic recipe for this algorithm. The RSA is thus
intimately connected both with the earlier proofs of (I.1) and with the basic
recursions satisfied by standard tableaux. 

In the course of our construction, however, we will need to make a number of
choices affecting the outcome of the final algorithm. By freely varying these
choices, we can produce not only the classical algorithm, but also continuum
many other algorithms yielding bijections proving (I.1), (of course, there are
only finitely many distinct ones for any fixed $n$). These bijections may have
applications in permutation enumeration. For example, it is well known [6] that
the longest increasing subsequence of a permutation is the length of the longest
row in the associated tableaux under the RSC. The new correspondences may permit
enumeration of permutations by other statistics. 

Cryptographic applications are also a possibility. Using any given algorithm to
obtain a pair of tableaux from a permutation is quite rapid. The large number of
similar algorithms, however, might form the basis for a secure code. As a
(perhaps naive) example, we note that the map defined by using one algorithm to
produce a pair of tableaux from a permutation, then using the inverse of a different
algorithm to make a codeword from the pair of tableaux, can profoundly scramble
the original permutation. 

In order to make this paper self-contained, we first review the definition of
the RSA, then give Young and Rutherford's algebraic proof of (I.1), and finally
present our algorithmic translation of this argument and sketch examples of the
new bijections. 

The authors are indebted to C.~Greene for helpful comments at the very early
stages of this work. 

\section {Background}
If $n$ is a positive integer, a partition $\lambda\vdash n$ is a decomposition
of $n$ in the form 
$$
n = \lambda_1 + \lambda_2 +\ldots+\lambda_k\;,
$$
where the $\lambda_i$ are positive integers and 
$$
\lambda_1\ge\lambda_2\ge\ldots\ge\lambda_k\;.
$$
A partition can be identified with a Ferrers diagram, a left-justified array of
squares which we shall write right side up: 
$$
(4, 3, 1) \longleftrightarrow
\begin{array}{cccc}
\square&&&\\
\square&\square&\square&\\
\square&\square&\square&\square\\
\end{array}
$$
A standard tableau (or Young tableau) of shape $\lambda\vdash n$ is a filling of
the squares of the Ferrers diagram of $\lambda$ with the integers $1, 2, \ldots, n$ in
which each integer appears exactly once and in which the entries increase in
each column from bottom to top and in each row from left to right. An example
for the partition $\lambda$ above is 
$$
\begin{array}{cccc}
5&&&\\
3&6&8&\\
1&2&4&7\\
\end{array}
$$
If $T$ is a standard tableau, then the shape of $T$, $\text{sh}(T)$, is the
partition of whose Ferrers diagram $T$ is a filling. The number of standard
tableaux of shape $\lambda$ is denoted $f_\lambda$. 

The Robinson-Schensted Algorithm yields a recursively defined bijection
$$
\sigma\longleftrightarrow (P, Q)
$$
between elements $\sigma\in S_n$ and pairs $(P, Q)$
of standard tableaux with 
$$
\text{sh}(P) = \text{sh}(Q) = \lambda\vdash n.
$$
 To describe the algorithm, we first show how to insert a letter $b$ into a row 
$$
(a_1 <a_2 <\ldots< a_m)
$$
not containing $b$. If $b > a_m$, $b$ is just appended to the row to obtain 
$$
(a_1 < a_2 <\ldots< a_m < b).
$$
If, on the other hand, $a_i < b < a_{i+1}$ for some $i$, then $b$ replaces
$a_{i+1}$ to produce the row 
$$
(a_1 < a_2 <\ldots< a_i < b < a_{i+2} <\ldots< a_m).
$$
The element $a_{i+1}$ is said to be bumped out of the row by this process. 

An element $b$ can be inserted by repeated row insertions into a tableau $T$
with distinct entries not including $b$. The number $b$ is first inserted into
the bottom row of $T$. If some number $b'$ is bumped during this insertion, then
$b'$ is inserted into the second row of~$T$ in the same way. Any number bumped
from the second row of~$T$ is inserted into the third row, and so on. The
process continues until one of the insertions results in an entry settling at
the end of a row. The sequence of squares occupied by the bumped numbers is
called the bump path of $b$. An example is shown below: 
\begin{spreadlines}{3\jot}
\begin{gather*}
\begin{pmatrix}
9\\8&10\\
5&6&12&14\\1&2&3&7&11&\leftarrow&4
\end{pmatrix}
\ \rightarrow\ 
\begin{pmatrix}
9\\8&10\\
5&6&12&14&&\leftarrow&7\\
1&2&3&4&11
\end{pmatrix}\\
\ \rightarrow\ 
\begin{pmatrix}
9\\
8&10&&&&\leftarrow&12\\
5&6&7&14\\
1&2&3&4&11
\end{pmatrix}
\ \rightarrow\ 
\begin{pmatrix}
9\\
8&10&12\\
5&6&7&14\\
1&2&3&4&11
\end{pmatrix}\\
\end{gather*}
\end{spreadlines}

Finally, to compute the tableaux $P$ and $Q$ associated with the permutation
$\sigma\in S_n$, the entries $\sigma_1, \sigma_2,\ldots,\sigma_n$ of $\sigma$
are inserted one at a time from left to right. The left tableau $P$ is the
result of these insertions; the entries in the right tableau~$Q$ keep track of where
a new square is added to the boundary of $P$ at each stage.  An example should
make this clear. The steps in the insertion of the permutation 2641375 are: 
\begin{spreadlines}{3\jot}
\begin{gather*}
(\varnothing\leftarrow2641375)
\ =\ 
\left(\begin{pmatrix}2&,&1\end{pmatrix}\leftarrow641375\right)
\ =\ 
\left(\begin{pmatrix}2&6&,&1&2\end{pmatrix}\leftarrow41375\right)\ =\\
\left(\begin{pmatrix}6& & &3& \\2&4&,&1&2\end{pmatrix}\leftarrow1375\right)
\ =\ 
\left(\begin{pmatrix}6& & &4& & \\2& &,&3& \\1&4&
&1&2\end{pmatrix}\leftarrow375\right)\ =\\
\left(\begin{pmatrix}6& & &4& & \\2&4&,&3&5\\1&3&
&1&2\end{pmatrix}\leftarrow75\right)
\ =\ 
\left(\begin{pmatrix}6& & & &4& & \\2&4& &,&3&5\\1&3&7
& &1&2&6\end{pmatrix}\leftarrow5\right)\ =\\
\begin{pmatrix}6& & & &4& & \\2&4& 7&,&3&5&7\\1&3&5& &1&2&6\end{pmatrix}\ =\ (P,Q)
\end{gather*}
\end{spreadlines}
Here and henceforth, $(P, Q) = (\varnothing\leftarrow\sigma)$ denotes the pair
of tableaux obtained by row inserting the permutation $\sigma$.

It is not difficult to show that this algorithm yields a bijection. 

The algorithm we have described is usually referred to as {\it row insertion.}
An alternative algorithm, column insertion, is obtained by transposing the
process, and inserting the letters of the permutation $\sigma$ from right to
left. Its relation to row insertion is discussed in the references by Knuth and
Schützenberger. 

Before giving the details of the algebraic proof of (I.1), we introduce a bit
more notation. If $\lambda\vdash n$ is a partition of $n$, we have used
$\lambda^-\vdash n-1$ and $\lambda^+\vdash n+1$ to denote generic partitions
obtained from  $\lambda$ by removing or adding a square from the Ferrers
diagram. Similarly, let $\lambda^{+-}$ be any partition of $n$ obtained by first
adding a square to the boundary of $\lambda$ and then removing a {\it
different\/} square from the boundary of the resulting shape. The same shapes
are obtained by first removing and then adding a square. The shape $\lambda$
itself is not one of the $\lambda^{+-}$. In the Young lattice, the lattice of
all Ferrers diagrams ordered by inclusion, $\lambda^-$ is any predecessor
of~$\lambda$, and $\lambda^+$ is any successor. A fragment of the Young lattice
showing $\Lambda^- = \{\lambda^-\}$, $\Lambda^+= \{\lambda^+\}$, and
$\Lambda^{+-} = \{\lambda^{+-}\}$ for the shape $\lambda= (3, 3, 2)$ is
shown in Fig.~1. Observe that every 
$\lambda^{+-}$ is the predecessor of a unique $\lambda^+$, and the successor of
a unique $\lambda^-$. 
    
\mypdf{0.75}{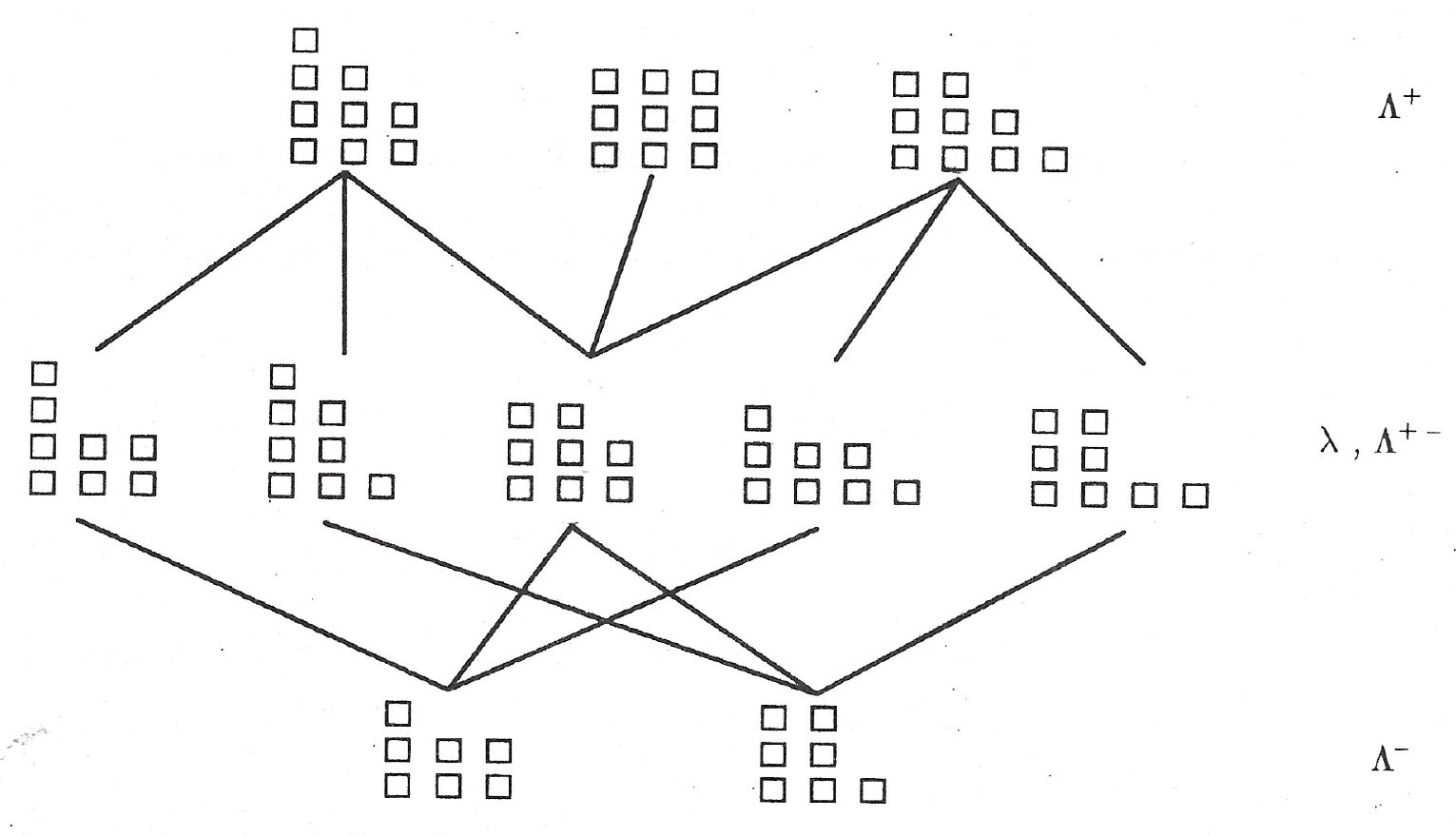}{The shape $\lambda=(3, 3, 2)$, the shapes
$\lambda^+$, $\lambda^-$, and $\lambda^{+-}$.}

Young's proof of (I.1) begins with the downward recursion 
\begin{equation}
f_\lambda=\sum_{\lambda^-}f_{\lambda^-}\;.\tag{1.1}
\end{equation}
The bijective proof of this recursion is trivial: removing the square containing
$n$ from a standard tableau of shape $\lambda$ produces a standard tableau whose
shape is some $\lambda^{-}$.

Young next proves the upward recursion,
\begin{equation}
(n + 1) f_\lambda = \sum_{\lambda^+}f_{\lambda^+}\;.\tag{1.2}
\end{equation}

The argument begins by observing that if $\lambda$ has $m$ predecessors in the
Young lattice, then it has $m +1$ successors. For the shape $(3, 3, 2)$ shown
above, $m= 2$. We call the $m +1$ positions at which a square can be added to
the boundary of $\lambda$ the {\it addible squares\/} of $\lambda$. The $m$
positions from which a square can be removed are the {\it removable squares\/}
of $\lambda$. For the shape $(3, 3, 2)$, these squares are shown below.
    
\mypdf{1.0}{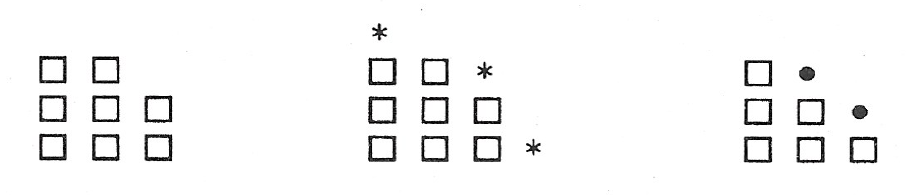}{The shape $\lambda$, its addible squares $*$, its
removable squares $\bullet$.}

With this notation, and keeping in mind Fig.~1, the proof of (1.2) is an easy
induction. If $\lambda\vdash n=1$, the upward recursion is trivial. Suppose that
it is true for all partitions of numbers smaller than $n$. The downward
recursion from $\lambda$ says that 
\begin{equation}
\sum_{\lambda^-}f_{\lambda^-}=f_\lambda\;.\tag{1.3}
\end{equation}
Summing the downward recursions from the shapes $\lambda^+$ yields 
\begin{equation}
(m+1)f_\lambda+\sum_{\lambda^{+-}}f_{\lambda^{+-}}
=\sum_{\lambda^+}f_{\lambda^+}\;.\tag{1.4}
\end{equation}
This is because each $\lambda^+$ has $\lambda$ as a predecessor, and each
$\lambda^{+-}$ has a unique $\lambda^+$ as a successor. Similarly, the sum of
the upward recursions from the shapes $\lambda^-$ is 
\begin{equation}
mf_\lambda+\sum_{\lambda^{+-}}f_{\lambda^{+-}}
=n\sum_{\lambda^-}f_{\lambda^-}\;.\tag{1.5}
\end{equation}
Combining formulas (1.3)--(1.5) then gives 
\begin{equation}
(n+1)f_\lambda=f_\lambda+n\sum_{\lambda^-}f_{\lambda^-}
=
(m+1)f_\lambda+\sum_{\lambda^{+-}}f_{\lambda^{+-}}
=
\sum_{\lambda^+}f_{\lambda^+}\;,\tag{1.6}
\end{equation}
completing the proof. 

Formula (I.1) follows at once from (1.1) and (1.2): 
\begin{equation}
(n+1)\sum_{\lambda\vdash n}f_\lambda^2
=
\sum_{\lambda\vdash n}\sum_{\lambda^+}f_{\lambda^+}\cdot f_\lambda
=
\sum_{\lambda\vdash (n+1)}\sum_{\lambda^-}f_{\lambda}\cdot f_{\lambda^-}
=
\sum_{\lambda\vdash(n+1)}f_\lambda^2\;.\tag{1.7}
\end{equation}
Since clearly $1!=\sum_{\lambda\vdash1}f_\lambda^2$, this proves (I.1). 

\section{Results}

How can the arguments in (1.6) and (1.7) be rendered bijective? To prove (1.2)
combinatorially, one needs a bijection 
\begin{equation}
\phi :(S, i) \longleftrightarrow T,\tag{2.1}
\end{equation}
where $T$ varies over all standard tableaux whose shape is in $\Lambda^+$, and
$i$ runs over all integers $1\le i\le n+1$. It will be convenient to let $S$
vary not over the standard tableaux of shape $\lambda$, but over all tableaux of
shape $\lambda$ whose entries are the integers $1, 2, \ldots,i-1, i
+1,\ldots,n+1$. These pairs $(S, i)$ are in obvious bijection with the pairs
$(S, i)$ where $S$ is standard. 

The bijection $\phi$, like row and column insertion, maps a tableau and an
integer to a new tableau with one more square than the old. For this reason, we
call $\phi$ an {\it insertion scheme,\/} and we speak of the evaluation of
$\phi(S, i)$ as the insertion of $i$ into ~$S$. 

There are three equalities in (1.6). The first comes from (1.3), the downward
recursion from $f_\lambda$. The next comes from (1.5), the upward recursion
applied to the numbers $f_{\lambda^-}$. The last comes from (1.4), the downward
recursion applied to the numbers $f_{\lambda^+}$. This suggests the following
rough recipe for computing $\phi$. Assume that $\phi$ has already been defined
whenever $S$ is a tableau of size less than $n$, and let $\text{sh}(S)
=\lambda\vdash n$. 

\begin{enumerate}

\item Given $(S, i)$ with $\text{sh}(S) =\lambda\vdash n$ and $1 \le i\le n+1$,
remove the largest element 
from $S$ to produce the tableau $S^-$.

\item Let $T^{-}=\phi(S^-, i)$. That is, use the upward recursion $\phi$ to
associate a tableau 
$T^-$ with the pair $(S^-, i)$. Since $\text{sh}(T^-)$ is a successor of
$\text{sh}(S^-)$ in the Young 
lattice, we have either $\text{sh}(T^-) \in\Lambda^{+-}$ or $\text{sh}(T^-)
=\lambda$. 

\item As long as $\text{sh}(T^-) \in\Lambda^{+-}$, produce $T$ by adding $n +1$
to the boundary of $T^-$ so 
as to make $\text{sh}(T) \in\Lambda^+$. Since every $\lambda^{+-}$ is contained
in exactly one $\lambda^+$, this can be done in only one way. 

\end{enumerate}

Removing the largest entry from a tableau was how we proved the downward
recursion (1.1), so the three steps in this outline correspond to the three
equalities in (1.6). 

In order to turn this outline into an exact algorithm for $\phi$, and thus to
prove (1.2) bijectively, we must settle two details. First, what happens if $i=
n +1$, in which case $\phi(S^-, i)$ makes no sense? Second, what if
$\text{sh}(T) =\lambda$? To answer these questions, we need to look more closely
at the arguments in Section 1. 

Formula (1.6) uses the fact that every Ferrers diagram has one more addible
square than removable square. By a {\it bumping scheme} we mean a family of
injections proving this fact, i.e., a set $\{\beta_\lambda: \lambda\vdash n, n
\in{\mathbb N}\}$ of injections, one for each shape, such that 
$$
\beta_\lambda : \{\text{removable squares of $\lambda$}\} \rightarrow
\{\text{addible squares of $\lambda$}\}
$$
Equivalently, $\beta_\lambda$ can be regarded as an injection $\beta_\lambda :
\Lambda^-\rightarrow\Lambda^+$. These injections can be chosen completely
arbitrarily. Under any bumping scheme, each $\lambda$ will have a unique addible
square not in the image of $\beta_\lambda$. This square is called the {\it lone
square\/} of~$\lambda$.

A simple bumping scheme is that in which any removable square is mapped by
$\beta_\lambda$ to the addible square one row higher. The lone square thus lies
at the end of the bottom row. This scheme, which will give rise to row
insertion, is illustrated in Fig.~3. 

\mypdf{1.0}{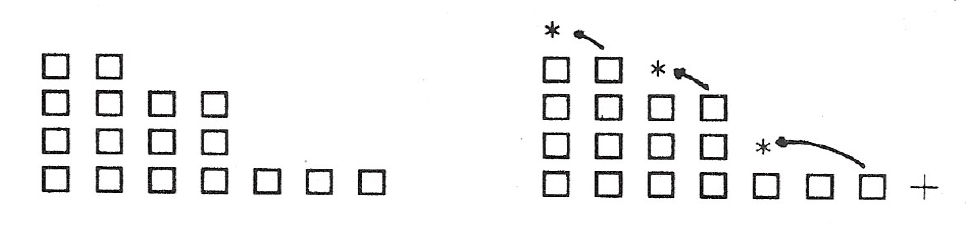}{A typical shape $\lambda$ (left) and the effect of the
injection $\beta_\lambda$ for the row insertion bumping scheme (right). Addible
squares are denoted $*$ and $+$ , the square marked $+$ being the lone square.}

A slightly more complex bumping scheme might take the bottom removable square of
any shape to the top addible square, the second lowest removable square to the
second highest addible square, and so on. For this scheme as well, the lone
square is always in the first row. Figure 4 shows a picture. 

\mypdf{1.0}{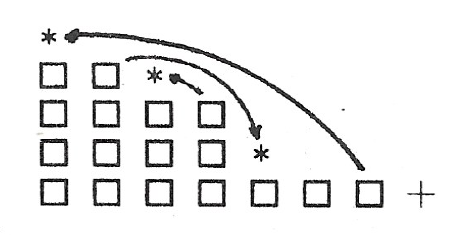}{An example of the injection $\beta_\lambda$ for a more
complicated bumping scheme.}

Since there is free choice of the injection $\beta_\lambda$ for each Ferrers
diagram $\lambda$, there are continuum many bumping schemes. 

We now fix a bumping scheme and return to fill in the details in our
construction of the bijection $\phi$. Given an integer $i$ such that $1\le i\le
n+1$ and a tableau $S$ of shape $\lambda\vdash n$ whose entries are
$1,\ldots,i-1, i +1,\ldots,n+1$, the tableau 
$T= \phi(S, i)$ is constructed as follows:

\begin{enumerate}

\item If $i=n+1$, $T$ is obtained by adding $n+1$ to $S$ in the lone square of
$\lambda$. Call 
this situation {\bf Case I\/}.
\par\noindent
If $i\ne n+1$, begin by removing the largest entry in $S$ to produce the
tableau~$S^-$.

\item If we are not in Case I, $T^- = \phi(S^-, i)$ is computed recursively. There
are now 2 more cases.

\item {\bf Case II:\/} If $\text{sh}(T^-) \ne\lambda$, then $\text{sh}(T^-) =
\lambda^{+-}$ is the predecessor of a unique 
shape $\lambda^+\in\Lambda^+$. The shape $\lambda^+$ is obtained by adding to
$\lambda^{+-}$ the unique square of $\lambda$ which is not in $\lambda^{+-}$. To
produce $T$, add $n +1$ to $T^-$ in this square.
\par\noindent
{\bf Case III:\/} Finally, suppose $\text{sh}(T^-) = \lambda$. In this case, let
$\lambda^+ = \beta_\lambda(\text{sh}(S^-))$. Produce $T$ from $T^-$ by adding
$n+1$ to the unique square in 
$\lambda^+/\lambda = \{\text{squares in $\lambda^+$ but not in $\lambda$}\}$. In
other words, place $n +1$ in the addible square of $\lambda$ associated by the
bumping scheme with the removable square which contained $n +1$ in $S$. Even
more informally, bump $n +1$ using $\beta_\lambda$.

\end{enumerate}

The purpose of this construction is to associate the $m$ copies of $f_\lambda$
on the left hand side of (1.5) with $m$ of the copies of $f_\lambda$ on the left
hand side of (1.4) in some explicit fashion. The last copy of $f_\lambda$ in
(1.4) comes from Case I. A bit of reflection shows that this algorithm is an
exact combinatorial rephrasing of (1.6). 

We will give examples of insertion schemes $\phi$ below, but first we prove two
propositions establishing that $\phi$ is a bijection and giving a more
computationally convenient description. 

\begin{prop}

$\phi$ is a bijection.

\end{prop}

\renewcommand{\qedsymbol}{{\bf QED}}

\begin{proof}

We explicitly compute 
$$
\phi^{-1}:T\longleftrightarrow(S, i)
$$
The idea of the construction is straightforward: one removes $n + 1$ from $T$,
applies $\phi^{-1}$ recursively to produce a pair $(S^-, i)$ with shape
$\text{sh} (S^-) \in\Lambda^-$, and then adds $n+1$ to the boundary of $S^-$ to
obtain the tableau $S$ of shape $\lambda$. The details look like this (steps are
numbered to agree with the steps in the definition of $\phi$):

\begin{enumerate}

\item[(3)] Beginning with a shape $\lambda$ and a standard tableau $T$ of some
shape $\lambda^+$, remove 
the square containing $n +1$ from $T$. Call the resulting standard tableau
$T^-$.

\end{enumerate}

\par\noindent
There are again three cases.

\smallskip

\par\noindent{\it Case I:\/} If $\text{sh}(T^-) = \lambda$ and $T$ is obtained
from $T^-$ by adding $n+1$ in the lone square, let $(S, i) = (T^-, n + 1)$.

\medskip

\par\noindent{\it Case II:\/} If $\text{sh}(T) \in\Lambda^{+-}$, then there is a
unique shape $\lambda^-\in\Lambda^-$ such that 
$\lambda^-\subseteq\text{sh}(T^-)$. With this shape $\lambda^-$, proceed as
follows:

\smallskip

\begin{enumerate}

\item[(2)] Use $\phi^{-1}$ recursively to map the tableau $T^-$, whose shape is
a successor of $\lambda^-$, to 
a pair $(S^-, i)$ with $\text{sh}(S^-) = \lambda^-$ and $1 \le i\le n$.

\item[(1)] Produce $S$ from $S^-$ by adding $n+1$ to the unique square in
$\lambda/\lambda^-$. It is not 
hard to see that this is the same square which contained $n +1$ in $T$. Both
elements of the pair $(S, i)$ are now defined. 

\end{enumerate}

\smallskip

\par\noindent
{\it Case III:} Finally, suppose that $\text{sh}(T^-)=\lambda$ and $\text{sh}(T)
\in \text{image} (\beta_\lambda)$. In other words, $T$ is not obtained from
$T^-$ by adding $n+1$ in the lone square. Let
$\lambda^-=\beta_\lambda^{-1}(\text{sh}(T))$, i.e., find the removable square of
$\lambda$ associated by the bumping scheme with the addible square in
$\text{sh}(T)/\lambda$. Remove this square from $\lambda$ to obtain $\lambda^-$,
and proceed just as above:

\begin{enumerate}

\item[($2'$)] Use $\phi^{-1}$ recursively to map $T^-$ to a tableau $S^-$ of
shape $\lambda^-$ and an integer~$i$ 
such that $1 \le i\le n$.

\item[($1'$)] Produce $S$ from $S^-$ by adding $n +1$ to the unique square in
$\lambda/\lambda^-$. In this case 
$n+1$ does not land in the same square it occupied in $T$. It lands in the
removable square of $\lambda$ associated by the bumping scheme to this addible
square. 

\end{enumerate}

It is straightforward to see that case (I) of this definition inverts case (I)
in the definition of $\phi$, case (II) inverts case (II), and case (III) inverts
case (III).

\end{proof}

\begin{prop}

$\phi(S,i)$ can be computed as follows:

\begin{enumerate}

\item Remove from the tableau $S$ all entries $j>i$; place $i$ in the lone
square of the resulting tableau. Call the result $T^i$.

\end{enumerate}

\par\noindent
Now construct a sequence $T^{i+1}, T^{i+2}, \ldots, T^{n+1}= T$ using these
rules:

\begin{enumerate}

\item[(2)] If the square containing $j$ in $S$ is vacant in $T^{j-1}$, produce
$T^j$ from $T^{j-1}$ by 
adding $j$ to this square.

\smallskip

\item If, on the other hand, the square containing $j$ in $S$ is occupied in
$T^{j-1}$, then it 
will be a removable square of $T^{j-1}$. Produce $T^j$ by adding $j$ to
$T^{j-1}$ in the addible square associated with this removable square by the
bumping scheme.

\end{enumerate}

\end{prop}

\begin{proof}

A trivial induction establishes that if $S_{\le j}$ is the tableau obtained from
$S$ by deleting all entries $>j$, then either $\text{sh}(T^j) = \text{sh}(S_{\le
j})$ or $\text{sh}(T^j) = \text{sh}(S_{\le j})^{+-}$. This guarantees that the
sequence in (2) and (3) is well-defined. It is easy to see that the algorithm
described here satisfies the recursive description given above for~$\phi$. In
particular, rule (1) corresponds to case (I) in the definition of $\phi$, rule
(2) corresponds to case (II), and rule (3) corresponds to case (III).

\end{proof}

As an example, let $\phi$ be the insertion scheme induced by the row insertion
bumping scheme of Fig.~3. The sequence of tableaux obtained in inserting $i= 3$
into the tableau 
\begin{equation}
S=\begin{array}{ccccc}
10&11& & & \\
4&5&8& & \\
1&2&6&7&9
\end{array}\tag{2.2}
\end{equation}
is shown in Fig.~5.

\mypdf{1.0}{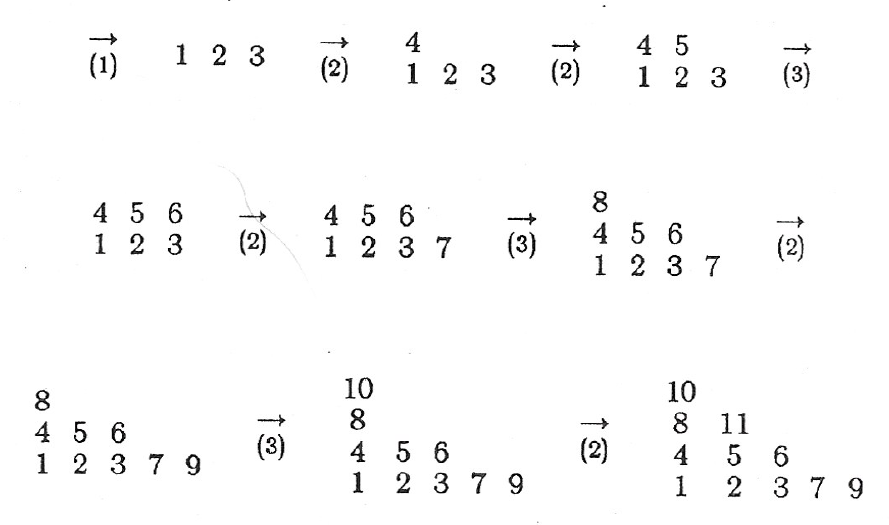}{The process of inserting 3 into the tableau $S$ of (2.2) using
the row insertion bumping scheme. The rule from Prop.~2 used at each stage is
shown under the appropriate arrow.}

Notice that the result of the insertion in Fig.~5 is the same as that of just
row inserting 3 into $S$. The numbers which were added in Fig.~5 using rule (3)
(the one involving bumping) are exactly those which are bumped to a higher row
during row insertion. We shall prove in a moment that the insertion scheme
engendered by the row insertion bumping scheme is indeed ordinary row insertion. 

As a more exotic example, consider again the bumping scheme in Fig.~4. The
sequence of tableaux obtained when 3 is inserted into the tableau $S$ of (2.2)
using this bumping scheme is shown in Fig.~6.

\mypdf{1.0}{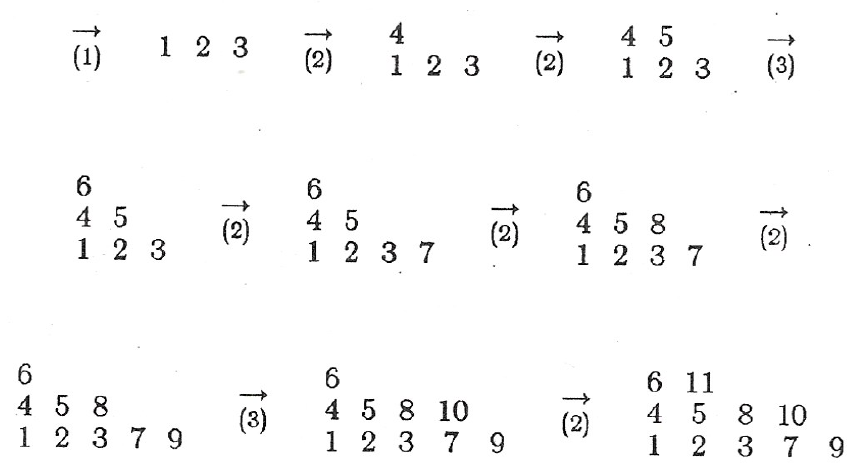}{The process of inserting 3 into the tableau $S$ of (2.2) using
the bumping scheme of Fig.~4. The rule used at each stage is shown under the
appropriate arrow.} 

The final tableau of Fig.~6 can be obtained from the initial tableau of (2.2) by
an algorithm resembling row insertion but having a different bump path. Most of
the entries of~$S$ are unchanged by the insertion; but the 3 has bumped the 6 up
to the top row, and the 6 has then bumped the 10 down to the second row. The
bumped numbers 
are those added using rule (3) of Prop.~2. This interpretation shows how close
all our insertion schemes are to row insertion: different bumping schemes just
result in different bump paths. 

In defining an insertion scheme, we inserted the letter $i$ into a tableau whose
entries are $1, 2, \ldots,i-1, i+ 1, \ldots, n +1$. In what follows, however, it
will be desirable to be able to insert $i$ into any tableau $S$ with distinct
entries not including $i$. This can be done by an obvious extension of the
algorithm of Prop.~2: one writes down all the entries of $S$ which are less than
$i$, puts $i$ in the lone square, then adds in increasing order the entries of
$S$ which are greater than $i$, bumping them if need be using 
$\{\beta_\lambda\}$. We use $\phi$ to denote this extended insertion scheme as
well. 

If this extension seems unmotivated by the algebra, one may adopt an alternative
approach in which integers are always inserted into {\it standard\/} tableaux. To insert
$i$, one first adds 1 to all entries of $S$ which are greater or equal to $i$,
and then inserts $i$ using Prop.~2. This approach is actually the direct
translation of Young's algebraic arguments, and is the one we originally
followed. The alternative approach taken here slightly simplifies the algorithms
and highlights their similarity to the classical RSA. 

\begin{prop}
If $\phi$ is the insertion scheme engendered by the row insertion bumping
scheme, then $\phi(S, i)$ is the tableau $(S \leftarrow i)$ obtained by row
inserting $i$ into $S$. 
\end{prop}

\begin{proof}

Let $S_{ab}$ be the element in the $a^\text{th}$ row (from the bottom) and
$b^\text{th}$ column of $S$, and assume to begin with that $S$ is a single row
$S_{11}, \ldots, S_{1n}$. To compute $\phi(S, i)$ using Prop.~2, we write all
entries of $S$ which are $\le i$, and we place $i$ at the end of this row. If
$S_{1, b-1} \le i< S_{1b}$, then $S_{1b}$ has just been displaced by $i$, so
$S_{1b}$ must be added to the tableau using rule (3). Since the row insertion
bumping scheme associates the removable square in the first row with the addible
square in the second row, this means that $S_{1b}$ is added as the leftmost
entry in row 2. Subsequent entries are added one at a time using rule (2), so
that they occupy the same positions in $T$ as in $S$. The resulting tableau $T$
is $(S \leftarrow i)$, so the proposition is proven in this case. 

Suppose now that the proposition is true for all tableaux with fewer rows than~$S$.
If $S_{1, b-1} \le i< S_{1b}$ , then inserting $i$ using Prop.~2 produces a
tableau whose first row begins $S_{11}, S_{12}, \ldots, S_{1, b-1} , i$.
Further, since bumped elements always move upward, the elements to the right of
$i$ in row 1 of $T$ must be exactly those of $S$. The first row of $\phi(S, i)$
is thus the first row of $(S \leftarrow i)$. In the remainder of the insertion
algorithm, 
$S_{1b}$ is added to the second row at the right of all elements which are
smaller than it, and insertion continues. This is exactly what would happen in
computing $\phi(S^*, S_{1b})$, where $S^*=S - \text{($1^{\text{st}}$ row)}$.
Since $S^*$ has fewer rows than $S$, however, $\phi(S^*, S_{1b}) = (S^*
\leftarrow S_{1b})$, and the proposition follows by induction.

\end{proof}

Proposition 3 shows that we have succeeded in producing from the downward
recursion (1.1) a family of bijections proving the upward recursion and
including row insertion. Column insertion is also included among these
bijections; its insertion scheme is the conjugate of that shown in Fig.~3. This
completes the combinatorial translation of~(1.6). 

We now turn our attention to the proof of (I.1) in (1.7). In order to prove
(I.1) combinatorially, we need a bijection 
$$
\rho:\sigma\longleftrightarrow (P, Q)
$$
between permutations and pairs of standard tableaux of the same shape. The
algebra in (1.7) tells us exactly how to compute $\rho$ recursively. The left
hand expression counts pairs $(P, Q)$ of tableaux together with an integer $i$.
The first equality says to insert $i$ into $P$ using $\phi$, i.e., to employ the
upward recursion (1.2). The second equality is just a relabeling. The last
equality, coming from the downward recursion (1.1), says to add a square
containing $n+1$ to $Q$ so as to produce a tableau of the same shape as the new
$P$ tableau. 

In short, given a permutation $\sigma$ one proceeds as in the ordinary RSA to
insert the letters of $\sigma$ one at a time (using the insertion scheme $\phi$)
to produce the $P$~tableau. The $Q$ tableau records the positions at which new
squares are added to the boundary of the $P$ tableau at each stage of this
algorithm. Since $\phi$ is a bijection, each stage of this algorithm is
invertible; so $\rho$ is a bijection as well. 

As an example, the figure below shows the tableaux produced when the permutation
$4265173$ is inserted one letter at a time using the bumping scheme of Fig.~4.

\mypdf{0.9}{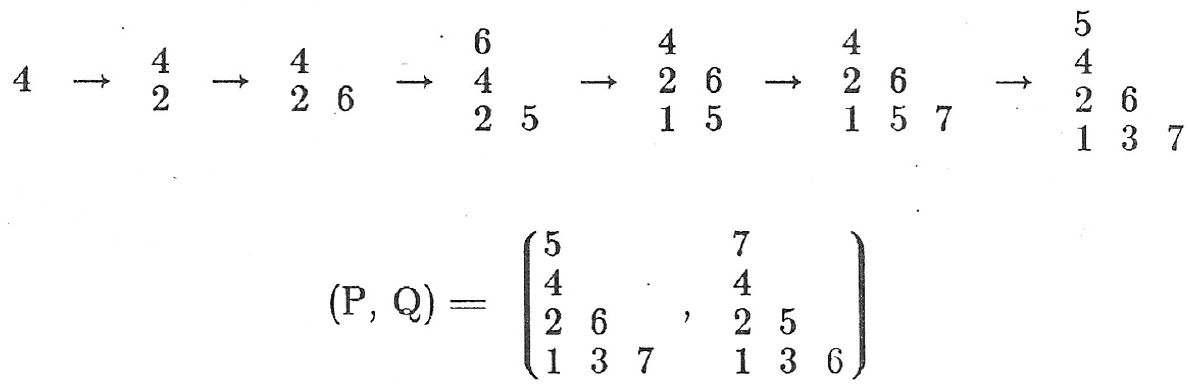}{The $P$ tableaux obtained by inserting the permutation
$4265173$ one letter at a time using the bumping scheme of Fig.~4. The final
$P$ and $Q$ tableaux are shown at bottom.}

With the construction of $\rho$ we complete our program of combinatorializing
Young's algebraic arguments. It is pleasant that this effort so easily
reconstructs the classical RSA, though it is startling that it produces so many
other related algorithms. We end with some remarks on these. 

By way of stressing the number of the correspondences $\rho$, we note that for
any standard tableau $P$ one can find a correspondence $\rho$ such that
$\rho(\text{id}) = (P, P)$. (When the identity is inserted, the new entry is
always added to the lone square, so $P$ must equal $Q$.) Indeed, in finding such
a $\rho$ one needs only specify the lone squares of the tableaux $P_{\le i}$; so
in fact a large number of correspondences produce a given target tableau from
the identity. 

Not all these correspondences are equally natural or convenient, however. For
example, the classical RSA extends to words other than permutations. It provides
a bijection between $\{1, 2, \ldots, n\}^*$ and the class of ordered pairs $(P,
Q)$ where $P$ and $Q$ have size $n$, $\text{sh}(P) = \text{sh}(Q)$, $Q$ is
standard, and $P$ is column-strict. (A column-strict tableau is one in which
entries increase strictly in columns and weakly in rows.) It is not clear, in
general, whether or not our algorithms admit such extensions.

Further properties of these algorithms, and other types of tableaux are
considered in [3] and later publications. 

\section*{References}

\begin{enumerate}

\item[{[1]}] Knuth, Donald E., ``The Art of Computer Programming, vol.~3,
Sorting and Searching,'' Reading, Mass., Addison-Wesley, 1973. 

\item[{[2]}] Knuth, Donald E., ``Permutations, Matrices, and Generalized Young
Tab-leaux,'' Pacific J.~Math., {\bf 34}, 709-727 (1970). 

\item[{[3]}] T. J. McLarnan, ``Tableaux recursions and symmetric Schensted
Correspondences for ordinary, shifted and oscillating tableaux'',
(U.C.S.D.~Thesis 1986) 

\item[{[4]}] Robinson, G.~de B., ``On the Representations of the Symmetric
Group,'' Amer.~J.~Math., {\bf 60}, 745-760 (1938). 

\item[{[5]}] Rutherford, Daniel Edwin, ``Substitutional Analysis,'' New York,
Haffner, 1968. 

\item[{[6]}] Schensted, C., ``Longest Increasing and Decreasing Subsequences,''
Canad.{} J.{} Math., {\bf 13}, 179-191 (1961). 

\item[{[7]}] Schützenberger, M.-P., ``La correspondance de Robinson,'' in
Dominique Foata, ed., “Combinatoire et représentation du groupe symétrique,”
Springer Lecture Notes no.~579, Berlin, Springer-Verlag, 1977. 

\item[{[8]}] Schützenberger, M.-P., ``Quelques remarques sur une construction de
Schensted,'' Math.~Scand., {\bf 12}, 117-128 (1963). 

\item[{[9]}] Young, Alfred, ``On Quantitative Substitutional Analysis III,''
Proc.{} London Math.{} Soc., {\bf 28}, 255-292 (1928) or in ``The Collected
Papers of Alfred Young,'' Mathematical Expositions no. 21, Toronto, University
of Toronto Press, 1977. 

\end{enumerate}

\end{document}